\documentclass[a4paper,12pt]{article}

\usepackage[english]{babel}

\usepackage[margin=2cm]{geometry}
\usepackage{type1cm}	
\usepackage{titlesec}   
\usepackage{fancyhdr}   

\usepackage[dvipsnames]{xcolor}

\usepackage{graphicx}   
\usepackage{titling}    
\usepackage{tabularx}   
\usepackage[shortlabels]{enumitem} 

\usepackage{amsmath, amsthm, amssymb} 
\usepackage{mathtools}  
\usepackage{thmtools} 
\usepackage{thm-restate}
    
\usepackage{tikz-cd}    
\usepackage{pgfplots}
\pgfplotsset{compat=1.18}
\usepackage{standalone} 
\usetikzlibrary{decorations.pathmorphing}
\usetikzlibrary{calc, arrows, matrix}
\usetikzlibrary{external}
\usepackage[unicode=true, pdfborder={0 0 0}, bookmarksdepth=-1]{hyperref} 
\usepackage[capitalise,nameinlink]{cleveref}

\newtheoremstyle{mystyle}
  {6pt}{15pt}
  {\it}
  {}
  {\bf}
  {.}
  {1em}
  {}

\theoremstyle{mystyle}	
\newtheorem{theorem}{Theorem}[section]
\newtheorem{example}[theorem]{Example}

\newtheorem{corollary}[theorem]{Corollary}
\newtheorem{property}[theorem]{Property}
\newtheorem{proposition}[theorem]{Proposition}
\newtheorem{lemma}[theorem]{Lemma}
\newtheorem{observation}[theorem]{Observation}

\newtheorem{conjecture}[theorem]{Conjecture}
\theoremstyle{definition}
\newtheorem{definition}[theorem]{Definition}

\newcommand{\N}{\mathbb{N}}

\newcommand{\E}{\mathbb{E}}
\newcommand{\F}{\mathcal{F}}
\renewcommand{\Pr}{\mathbb{P}}

\DeclareMathOperator{\Freq}{Freq}

\begin{document}

\title{Frequent elements in union-closed set families}
\author{Shagnik Das\thanks{Department of Mathematics, National Taiwan University, Taipei, Taiwan.  Research supported by Taiwan NSTC grants 111-2115-M-002-009-MY2 and 113-2628-M-002-008-MY4. Email: \texttt{shagnik@ntu.edu.tw}} \and Saintan Wu\thanks{
Department of Mathematics, National Taiwan University, Taipei, Taiwan. Email: \texttt{r12221025@ntu.edu.tw}}}
\date{\today}
\maketitle

\begin{abstract}
    The Union-Closed Sets Conjecture asks whether every union-closed set family $\F$ has an element contained in half of its sets. In 2022, Nagel posed a generalisation of this problem, suggesting that the $k$th-most popular element in a union-closed set family must be contained in at least $\frac{1}{2^{k-1} + 1} |\F|$ sets.  
    
    We combine the entropic method of Gilmer with the combinatorial arguments of Knill to show that this is indeed the case for all $k \ge 2$, and characterise the families that achieve equality. Furthermore, we show that when $|\F| \to \infty$, the $k$th-most frequent element will appear in at least $\left( \frac{3 - \sqrt{5}}{2} - o(1) \right) |\F|$ sets, reflecting the recent progress made for the Union-Closed Set Conjecture.
\end{abstract}

\section{Introduction} \label{sec:intro}

Frankl's Union-Closed Set Conjecture is arguably one of the most famous (or, according to some sources, notorious) open problems in extremal set theory. A family $\mathcal{F}$ of subsets of a ground set $[n]$ is said to be \emph{union-closed} if, for every $A, B \in \mathcal{F}$, we have $A \cup B \in \mathcal{F}$. While union-closed families come in many shapes and sizes, Frankl conjectured in 1979 that they all must contain a popular element.

\begin{conjecture} \label{conj:UCSC}
    For every union-closed set family $\mathcal F \neq \left\{ \emptyset \right\}$ over a ground set $[n]$, there is an element $i \in [n]$ belonging to at least half of the sets in $\mathcal{F}$. 
\end{conjecture}

Since its formulation, the conjecture has attracted a great deal of attention from the combinatorial community, even being the focus of a recent PolyMath project~\cite{PolyMath}. The problem has stubbornly resisted all attacks, with only partial results having been obtained to date.

In one direction, researchers have proven the conjecture for several special classes of union-closed set families. For example, Vu\v{c}kovi\'{c} and \v{Z}ivkovi\'{c}~\cite{vuckovic201712case} established the conjecture for union-closed set families when $n \le 12$, Roberts and Simpson~\cite{MR2662546} showed that a minimal counterexample requires $|\F| \ge 4n -1$, and Balla, Bollob\'as, and Eccles~\cite{MR3007135} resolved the case when $|\F| \ge \frac23 2^n$. Various other conditions have also been considered~\cite{MR4245296, MR3266293, MR2199779, MR1750455, MR1932685}, and we refer the reader to Bruhn and Schaudt's~\cite{MR3417215} comprehensive survey detailing progress along these lines.

Another approach has been to prove weaker results for general union-closed set families, showing that there must be an element in many sets, if not quite half of them. For instance, Knill~\cite{knill1994graph} showed that there is always an element contained in at least $\frac{|\mathcal{F}|-1}{\log_2|\mathcal{F}|}$ sets of a union-closed set family $\mathcal{F}$, with the constant factor later improved by W\'ojcik~\cite{MR1675919}. Other lower bounds involved the size of the ground set $n$; Balla~\cite{balla2011minimum} proved a lower bound of $\frac1{2} \left(\frac{\log_2n}{n} \right)^{1/2}|\mathcal{F}|$, while Reimer~\cite{MR1967488} gave a bound of $\frac{\log_2 |\mathcal{F}|}{2n} |\mathcal{F}|$, and these were the state-of-the-art until a few years ago.

In 2022, Gilmer~\cite{gilmer2022constant} made a major breakthrough, introducing novel entropic methods to come within a constant factor of the conjecture. He proved that if $\mathcal{F} \neq \{\emptyset\}$ is union-closed, it contains an element in $\tfrac{1}{100}|\mathcal{F}|$ sets. This sparked a flurry of activity, and within days several authors had independently optimised the calculations to push the bound further. Alweiss, Huang and Sellke~\cite{alweiss2022improved} and Pebody~\cite{pebody2022extension} improved the lower bound to $\frac{3-\sqrt{5}}{2}|\mathcal{F}|$, which Gilmer had suggested to be the limit of his method. Then Sawin~\cite{sawin2023improved}, Cambie~\cite{cambie2022better}, and Liu~\cite{liu2023improving} made further progress, increasing the bound from $\frac{3-\sqrt{5}}{2}|\mathcal{F}|\approx 0.381966 |\mathcal{F}|$ to approximately $0.3823455|\mathcal{F}|$, which is where things stand today. 
    
\subsection{Less frequent elements}

Aside from attempts to solve \cref{conj:UCSC} directly, there have also been several equivalent formulations and variants proposed and studied. In his note on union-closed set families, Nagel~\cite{nagel2023notes} suggested investigating the frequencies of elements beyond the most popular one, offering the following conjecture.

\begin{conjecture} \label{conj:nagel}
    For any union-closed set family $\mathcal{F}$ with $|\cup_{F \in \mathcal{F}} F| \ge k$, the $k$th-most frequent element lies in at least $\frac{|\F|}{2^{k-1}+1}$ sets in $\mathcal{F}$.
\end{conjecture}

Note that this generalises \cref{conj:UCSC}, which is the case $k = 1$. Furthermore, the conjectured bound is best possible, as evidenced by the families we call \emph{near-$k$-cubes}, which are complete boolean lattices of dimension $k-1$, together with one additional set; that is, $\mathcal{F}_k = 2^{[k-1]}\cup \{S\}$, where $[k-1] \subsetneq S$. This is a union-closed family of size $2^{k-1} + 1$ in which the element $k$ features only once. Unlike in Frankl's Conjecture, though, this is essentially the only known tight construction; we do not know of any constructions of larger sizes.

Before proceeding to our main result, we observe that Nagel's conjecture is in fact equivalent to Frankl's. This extends a remark of Nagel~\cite{nagel2023notes}, who used the Union-Closed Set Conjecture to give a weaker lower bound of $\frac{m}{2^k}$.

\begin{observation} \label{obs:equivalence}
    Assuming the Union-Closed Sets Conjecture, Nagel's conjecture is true for all $k$.
\end{observation}

\begin{proof}
    Fix $k \in \mathbb{N}$, and let $\mathcal{F}$ be a union-closed set family with $|\mathcal{F}| = m$ and $|\cup_{F \in \mathcal{F}} F| \ge k$. Without loss of generality, we suppose the $(k-1)$ most frequent elements in the ground set $[n]$ are $1, 2, \hdots, k-1$, breaking any ties arbitrarily.
    
    Consider the $2^{k-1}$-to-$1$ map $\pi_{k-1}: 2^{[n]} \to 2^{[n]\setminus [k-1]}$ defined by $\pi_{k-1}(F) = F \setminus [k-1]$. Observe that $\pi_{k-1}(\mathcal{F})$ is still a union-closed family that contains a nonempty set, since the support of $\mathcal{F}$ is too large to be contained in $[k-1]$.
    
    By the Union-Closed Sets Conjecture, there exists an element $a\notin [k-1]$ contained in ${f \ge \frac12 | \pi_{k-1}(\mathcal{F})|}$ sets in $\pi_{k-1}(\mathcal{F})$. Tracing sets in $\pi_{k-1}(\mathcal{F})$ back to their disjoint preimages in $\mathcal{F}$, we find the following hold:
    
    \begin{itemize}
        \item[(i)] if $F\in \pi_{k-1}(\mathcal{F})$ contains $a$, then $\pi_{k-1}^{-1}(F) \subseteq \mathcal{F}$ contains at least one set, and they all contain $a$, and
        \item[(ii)] if $F\in \pi_{k-1}(\mathcal{F})$ does not contain $a$, then $\pi_{k-1}^{-1}(F) \subseteq \mathcal{F}$ contains at most $2^{k-1}$ sets, none of which contain $a$.
    \end{itemize}
    
    It follows that the proportion of sets in $\mathcal{F}$ containing $a$ is at least $\frac{f}{f+2^{k-1}(|\pi_{k-1}(\mathcal{F})|-f)} \geq \frac{1}{1+2^{k-1}}$, and hence Nagel's conjecture is true.
\end{proof}

\subsection{Our results}

Nagel~\cite{nagel2023notes} himself had proven the conjecture unconditionally in the special cases of the least and second-least frequent elements. Our main result provides a complete resolution of \cref{conj:nagel} for all $k \ge 2$.

\begin{restatable}{theorem}{mainresult}
\label{thm:main}
    Let $k \geq 2$, and let $\mathcal{F}$ be a union-closed set family with $|\cup_{F \in \mathcal{F}} F| \ge k$. Then the $k$th-most frequent element lies in at least $\frac{|\F|}{2^{k-1} + 1}$ sets in $\mathcal{F}$, with equality only if $\mathcal{F}$ is a near-$k$-cube.
\end{restatable}

Our proof combines the approaches previously used to attack \cref{conj:UCSC}. For large set families, we adapt the entropic argument of Gilmer~\cite{gilmer2022constant} to establish lower bounds on the frequency of the $k$th-most frequent element. In fact, as shown in the following theorem, we can match the Gilmeresque lower bound, suggesting that large union-closed set families do not exhibit any drop-off in the frequencies of the most popular elements.
    
\begin{restatable}{theorem}{largefamilies}
\label{thm:largefamilies}
    For any $0\leq \alpha<\frac{3-\sqrt{5}}{2}$, there is a constant $c_\alpha\geq 0$ such that if $k \ge 2$ and $\F$ is a union-closed set family with $|\F| \geq 2^{c_\alpha (k-1)}$, then there are at least $k$ elements in the ground set that each appear in at least $\alpha |\mathcal{F}|$ sets in $\mathcal{F}$.
\end{restatable} 

We remark that the proof gives an explicit value for $c_{\alpha}$, which yields a concrete lower bound on the size of families for which \cref{thm:largefamilies} applies. For smaller families, we instead use the combinatorial method of Knill~\cite{knill1994graph} to obtain the required lower bound on the frequency of the $k$th-most popular element.

\paragraph{Organisation and Notation.}
In \cref{sec:large} we prove \cref{thm:largefamilies}, establishing our bounds for large set families. We then study smaller set families in \cref{sec:small}. Finally, we bring the parts of our proof of \cref{thm:main} together in \cref{sec:conclusion}, before raising questions about the true nature of the $k$th-highest frequency in set families and suggesting directions for further research.

We shall take $[n] = \{1, 2, \hdots, n
\}$ to be the ground set for our set families $\mathcal{F}$, and will denote their size by $m = |\mathcal{F}|$. For an element $i \in [n]$ of the ground set, we denote its (relative) frequency in $\mathcal{F}$ by
\[ \Freq_{\mathcal{F}}(i) = \frac{|\{ F \in \mathcal{F}: i \in F\}|}{|\mathcal{F}|}, \]
and define the \emph{$k$th frequency} of $\mathcal{F}$, denoted $f_k(\mathcal{F})$, to be the $k$th-highest frequency among the elements of the ground set. We shall further assume, without loss of generality, that the elements are ordered by their frequency in $\mathcal{F}$; that is, $\Freq_{\mathcal{F}}(1) \ge \Freq_{\mathcal{F}}(2) \ge \hdots \ge \Freq_{\mathcal{F}}(n)$, and $f_k(\mathcal{F}) = \Freq_{\mathcal{F}}(k)$.

With this notation, Nagel's conjecture can be reformulated to say that for any union-closed set family $\mathcal{F}$ involving at least $k$ elements, we have $f_k(\mathcal{F}) \ge \frac{1}{2^{k-1} + 1}$. Finally, note that we may, and will, assume that the empty set is a member of $\mathcal{F}$, since, if it is not, including it preserves the union-closed property while decreasing the frequency of every element.

\section{Large families} \label{sec:large}

In this section we will prove \cref{thm:largefamilies}, showing that the Gilmeresque bounds on the most frequent element also apply to the $k$th-most frequent element when the family is large. In particular, we will use the method of entropy, and when taking a random set $A$, it will be useful to consider events of a given element being contained in $A$.

\begin{definition}
    For any distribution $A$ on $2^{[n]}$, let $A_i$ be the indicator variable of the event $i\in A$, and $A_{< i}$ be the joint variable $(A_1,A_2,\dotsc,A_{i-1})$. Note that $A = A_{< n+1}$.
\end{definition}

\subsection{Entropic preliminaries}

Entropy is a very useful tool derived from Information Theory.  Informally speaking, it quantifies the amount of information that revealing a random variable yields. Before presenting our proof, we provide an overview of some basic facts about entropy, focusing on the case of discrete random variables. We refer the interested reader to~\cite{cover1999elements} for further details.

\begin{definition}[Entropy]
    Given random variables $X$ and $Y$,
    \begin{itemize}
        \item[(i)] the \emph{entropy} of a discrete random variable $X$ is given by 
        \[H(X) = \sum_{x} -\Pr(X=x) \log_2(\Pr(X=x)),\]
        with the convention that $0 \log_2 0 = 0$, and
        
        \item[(ii)] the \emph{conditional entropy} of $X$ given $Y$ is
        $H(X|Y) = \E_y[H(X|Y=y)]$, the expected value of the entropy of $X$ given knowledge of $Y$. 
    \end{itemize}
\end{definition}

A simple yet important example is the entropy of the Bernoulli distribution.
\begin{example}[Binary entropy function]
    Let $X$ have the Bernoulli distribution with parameter $p$. Then $H(X)=-p\log_2p -(1-p)\log_2(1-p)$. Viewed as a function of $p$, denoted by $H(p)$, it is called the \emph{binary entropy function}.
\end{example}

In this paper, we will use the following fundamental properties of entropy, most of which can be proven by Jensen's inequality.

\begin{property} \label{property:entropy}
    Let $X, Y, X_0, X_1, \dotsc, X_n$ be random variables on a finite sample space.
    \begin{enumerate}
        \item \textbf{Chain rule.} $H(X,Y)=H(X|Y)+H(Y)$.\\
        More generally, $H(X_1,\dotsc,X_n) = \sum_{i=1}^{n} H(X_i| X_{<i})$. 
        
        \item \textbf{Range of entropy.}
        If $X$ takes $n$ values, we have $0 \le H(X)\leq \log_2 n$, and equality holds in the upper bound if and only if $X$ is uniformly distributed.
        
        \item \textbf{Conditioning lowers entropy.} $H(X|Y) \leq H(X)$, and equality holds if and only if $X$ and $Y$ are independent.
        
        \item \textbf{Data processing.} For any function $f$ of $Y$, we have $H(X|f(Y)) \ge H(X|Y)$. Intuitively, this means that knowing less in advance makes any new information $X$ more informative.
    \end{enumerate}    
\end{property}

\subsection{Size bound of families with low frequencies}
	
    We first briefly explain the idea behind Gilmer's proof. Gilmer~\cite{gilmer2022constant} proved that if we have a distribution over $2^{[n]}$ where every element in $[n]$ appears with probability  at most $\alpha$, for some $\alpha \le \frac{1}{100}$, then, for $A$ and $B$ sampled independently from this distribution, $H(A \cup B) \ge H(A)$. He proved this by using the chain rule to decompose the entropy element-wise, and then showing that the inequality 
    \[ H((A\cup B)_i|A_{< i},B_{<i})\geq H(A_i|A_{< i})\]
    holds for each element appearing infrequently. 
    
    Subsequent work~\cite{alweiss2022improved,sawin2023improved} showed that this also holds for $\alpha$ as large as $\frac{3-\sqrt{5}}{2}$. Furthermore, Sawin~\cite{sawin2023improved} proved the following optimisation result sharpening this inequality, which will be of use in our own work.
    \begin{lemma}\label{lemma:afterk}
        Let $A,B$ be i.i.d. random sets on the family $2^{[n]}$. Suppose further that $\E[A_i]\leq \alpha< \frac{3-\sqrt{5}}{2}$. Then 
        \[H((A\cup B)_i|A_{< i},B_{<i}) \geq \lambda_\alpha H(A_i |A_{< i}),\]
        where $\lambda_\alpha =\frac{H(2\alpha-\alpha^2)}{H(\alpha)}>1$.\footnote{Sawin also showed that this inequality holds when $\alpha \ge \frac{3 - \sqrt{5}}{2}$, but with the constant $\lambda_{\alpha}$ defined to be $\frac{1+\sqrt{5}}{2} (1 - \alpha)$ instead of $\frac{H(2\alpha - \alpha^2)}{H(\alpha)}$.} In particular, $H(A\cup B)\geq \lambda_\alpha H(A)$.
    \end{lemma}

To obtain lower bounds for the Union-Closed Set Conjecture, Gilmer considered the uniform distribution $\mathrm{Unif}(\F)$ over a union-closed set family. If all elements have low frequency in $\F$, then, for $A$ and $B$ sampled independently and uniformly from $\F$, we have $H(A \cup B) \ge H(A)$. However, as $\F$ is union-closed, it follows that $A \cup B$ is also a distribution over $\F$, but is not uniform --- the probability that $A \cup B$ is the empty set is only $1/|\F|^2$, since we need both $A$ and $B$ to be empty. This contradicts~\cref{property:entropy}, which asserts that the uniform distribution is the unique distribution maximising the entropy.

When dealing with the $k$th frequency, the condition $\E[A_i] \le \alpha$ now only applies for $i \ge k$. Thus, to make use of~\cref{lemma:afterk}, we shall first project $\F$ onto $[n] \setminus [k-1]$, thereby removing the $k-1$ most frequent elements. Afterwards, we shall perform some estimates to recover information about the original family.

Specifically, we prove the following.
    \begin{theorem}\label{thm:largeNew}
        Let $k\geq 2$ and $0 < \alpha<\frac{3-\sqrt{5}}{2}$. If $\F$ is a union-closed family with $f_k(\F)\leq \alpha$, then $\log_2|\F|\leq \frac{\lambda_\alpha}{\lambda_\alpha-1}(k-1)$. In fact,
        \[\log_2|\F| \leq \frac{\lambda_\alpha}{\lambda_\alpha-1}\cdot \frac{2^{k-1}(k-1)}{2^{k-1}-1} -\frac{1}{(\lambda_\alpha-1)}\log_2\left(\frac{\lambda_\alpha2^{k-1}(k-1)e}{(2^{k-1}-1)\log_2 e }\right).\]
    \end{theorem}
	\begin{proof}
            Let $\F$ be a union-closed family with $f_k(\F) \le \alpha$. As we do not have information about the frequencies of the $k-1$ most frequent elements, we shall project onto their complement, using the map $\pi_{k-1}$ from the proof of~\cref{obs:equivalence}:
	\begin{align*}
		\pi_{k-1}: 2^{[n]}&\to 2^{[n]\setminus[k-1]}\\
		F&\mapsto F\setminus [k-1].
	\end{align*}
	
        Now let $X$ be a uniformly random set in $\F$, and let $A = \pi_{k-1}(X)$ be its projection. Observe that for any $i \ge k$, we have $\E[A_i] = \Pr(i \in A) = \Pr(i \in X) \le \alpha$. Hence, if we let $B$ be independent of and identically distributed as $A$,~\cref{lemma:afterk} gives $H(A\cup B)\geq \lambda_\alpha H(A)$, where $\lambda_{\alpha}=\frac{H(2\alpha-\alpha^2)}{H(\alpha)}$. To obtain the desired contradiction, we need provide an upper bound for $H(A\cup B)$ and a lower bound for $H(A)$. 	
	
	For the upper bound, we simply apply the support bound. Since $\pi_{k-1}(\F)$ is still union-closed, $A\cup B$ is again a distribution on $\pi_{k-1}(\F)$. Hence, $H(A\cup B)\leq \log_2|\pi_{k-1}(\F)|$. 
	
	For a lower bound on $H(A)$, observe that, since $X$ is uniformly distributed over $\F$, we have $H(X) = \log_2 |\F|$. On the other hand, we have $H(X)=H(A,X)$, since $A$ is determined by $X$. Applying the chain rule gives $H(A,X)=H(A)+H(X|A)$. Putting this all together yields $H(A)=\log_2|\F|-H(X|A)$.

    Combining the upper and lower bounds results in
    \begin{equation} \label{eqn:entropy}
        \log_2 |\pi_{k-1}(\F)| \ge \lambda_\alpha \left( \log_2 |\F| - H(X|A) \right).
    \end{equation}
    
    We next evaluate $H(X|A)$. By definition, 
    \[ H(X|A) = \sum_{F \in \pi_{k-1}(\F)} \Pr(A = F)H(X|A = F) = \sum_{F \in \pi_{k-1}(\F)} \frac{|\pi_{k-1}^{-1}(F)|}{|\F|} H(X|A=F). \] 
    Conditioning on $A = F$, we have that $X$ is uniformly distributed on $\pi_{k-1}^{-1}(F)$, and so 
    \[ H(X|A) = \sum_{F \in \pi_{k-1}(\F)} \frac{|\pi_{k-1}^{-1}(F)|}{|\F|} \log_2 |\pi_{k-1}^{-1}(F)|. \]

    For a simple bound, observe that $|\pi_{k-1}^{-1}(F)| \le 2^{k-1}$ for all $F$, and so it follows that $H(X|A) \le k-1$. Substituting this into~\eqref{eqn:entropy}, together with the trivial bound $|\pi_{k-1}(\F)| \le |\F|$, we get $\log_2 |\F| \le \frac{\lambda_{\alpha}}{\lambda_{\alpha} - 1} (k-1)$, proving the first inequality from the statement of the theorem.

    For the more precise estimate, note that the function $f(x) = x \log_2 x$ is convex, and hence for any $x \in [1, 2^{k-1}]$, we have $f(x) \le \frac{2^{k-1}-x}{2^{k-1} - 1} f(1) + \frac{x - 1}{2^{k-1}-1} f(2^{k-1}) = (x-1)\frac{2^{k-1}(k-1)}{2^{k-1}-1}$. Thus, $H(X|A) = \sum_{F \in \pi_{k-1}(\F)} \frac{f(|\pi_{k-1}^{-1}(F)|)}{|\F|} \le \sum_{F \in \pi_{k-1}(\F)} \frac{|\pi_{k-1}^{-1}(F)|-1}{|\F|} \cdot \frac{2^{k-1}(k-1)}{2^{k-1} - 1} = \left( 1 - \frac{|\pi_{k-1}(\F)|}{|\F|} \right) \frac{2^{k-1}(k-1)}{2^{k-1}-1}.$

    Letting $\rho = \frac{|\pi_{k-1}(\F)|}{|\F|}$, we can substitute this into~\eqref{eqn:entropy} to obtain
    \[ \log_2 \rho + \log_2 |\F| \ge \lambda_\alpha \left( \log_2 |\F| - (1 - \rho) \frac{2^{k-1}(k-1)}{2^{k-1} - 1} \right), \]
    which can be rearranged to give $(\lambda_\alpha - 1) \log_2 |\F| \le \log_2 \rho + (1 - \rho) \lambda_\alpha \frac{2^{k-1}(k-1)}{2^{k-1}-1}$.
    
    By differentiating, we find the right-hand side is maximised when $\rho = \frac{(2^{k-1}-1) \log_2 e}{\lambda_\alpha 2^{k-1} (k-1)}$, for which the right-hand side becomes $\frac{\lambda_\alpha 2^{k-1}(k-1)}{2^{k-1}-1} - \log_2 \left( \frac{\lambda_\alpha 2^{k-1}(k-1)e}{(2^{k-1}-1)\log_2 e}\right)$. This gives the desired upper bound on $\log_2 |\F|$.
    \end{proof}

Using this result, we can prove \cref{conj:nagel} for large families.

\begin{proposition}\label{prop:large}
    Given $k \ge 2$, let $\mathcal{F}$ be union-closed with $|\cup_{F \in \F} F| \ge k$.
    \begin{itemize}
        \item[(a)] If $|\F|\geq 2^{2.71(k-1)}$, then $f_k(\mathcal{F})> \frac{1}{17}$.
        \item[(b)] If $k = 4$ and $|\mathcal{F}|\geq 14$, then $f_4(\mathcal{F})> \frac19$.
        \item[(c)] If $k = 3$ and $|\mathcal{F}|\geq 6$, then $f_3(\mathcal{F})> \frac15$.
        \item[(d)] If $k = 2$ and $|\mathcal{F}|\geq 4$,  then $f_2(\mathcal{F})> \frac13$.
    \end{itemize}
\end{proposition}

\begin{proof}
    For part (a), we apply~\cref{thm:largeNew} with $\alpha = \frac{1}{17}$. Note that $\lambda_{\frac1{17}}\approx 1.587624$, and so $\frac{\lambda_{\frac1{17}}}{\lambda_{\frac1{17}}-1}(k-1)< 2.71(k-1)$. Hence, our claim follows from the first inequality in the theorem.
    
    For the other parts, we use the second inequality in~\cref{thm:largeNew}. Let $B(\alpha,k)$ denote the upper bound on $\log_2 |\F|$ it gives. By direct evaluation, we have $B(\frac19,4) \approx 3.805781$, $B(\frac15,3) \approx 2.512253$, and $B(\frac13, 2) \approx 1.697618$, which yield the desired bounds on $|\F|$.
\end{proof} 

\section{Small families} \label{sec:small}

\cref{prop:large} resolves \cref{conj:nagel} for families $\F$ with $|\F| = 2^{\Omega(k)}$. In this section we handle smaller families, starting with those that are very small.

\begin{proposition} \label{prop:small}
    If $\F$ is a union-closed set family with $|\F| \le 2^k + 1$ and $|\cup_{F \in \F} F| \ge k$, then $f_k(\F) \ge \frac{1}{2^{k-1} + 1}$, with equality if and only if $\F$ is a near-$k$-cube.
\end{proposition}

\begin{proof}
    Recall that we order the elements by frequency, so in particular $[k-1]$ consists of the $k-1$ most frequent elements.
    
    First suppose $\mathcal{F}$ contains two sets $F_1, F_2$ that are not fully contained in $[k-1]$. Without loss of generality, $F_2 \not\subseteq F_1$. Then, setting $F_3 = F_1 \cup F_2$, we have $F_3 \neq F_1$, and they contain a common element outside $[k-1]$. This implies $f_k(\F) \ge \frac{2}{|\F|} \ge \frac{2}{2^k + 1} > \frac{1}{2^{k-1} + 1}$.

    Thus, $\F$ can only contain a single set $S \notin 2^{[k-1]}$, and so $|\F| \le 2^{k-1} + 1$ and $f_k(\F) = \frac{1}{|\F|}$. In particular, $f_k(\F) \ge \frac{1}{2^{k-1} + 1}$, and for equality, we must have $|\F| = 2^{k-1} + 1$. This is only possible if $2^{[k-1]} \subseteq \F$ and, furthermore, $[k-1] \subsetneq S$, as otherwise $S \cup [k-1]$ would be a second set in $\F \setminus 2^{[k-1]}$. In other words, $\F$ is a near-$k$-cube.
\end{proof}

Observe that between~\cref{prop:large} and~\cref{prop:small}, we have already proven~\cref{conj:nagel} when $2 \le k \le 4$. For $k \ge 5$, it remains to handle set families $\mathcal{F}$ with $2^k + 2 \le | \mathcal{F} | < 2^{2.71(k-1)}$. For these families, we adapt a proof of Knill~\cite{knill1994graph} who, we recall, proved that $f_1(\F) \ge \frac{1}{\log_2 |\F| + 1 }$. In our range of interest, $\log_2 |\F| = \Theta(k)$ is much smaller than the denominator of $2^{k-1} + 1$ in \cref{conj:nagel}, and so although Knill's bound is a considerable weakening of Frankl's Conjecture, it will be more than sufficient for our purposes.

Our goal is to show that there is some element in $[n] \setminus [k-1]$ with large frequency. To this end, we denote by $\F_{\ge k}$ the subfamily of $\F$ consisting of sets that contain some element in $[n] \setminus [k-1]$; that is, $\F_{\ge k} = \F \setminus 2^{[k-1]}$. We are assuming $|\cup_{F \in \F} F| \ge k$, which implies $\F_{\ge k} \neq \emptyset$, and we shall be particularly interested in vertex covers of this subfamily.

\begin{definition}[$k$-good]
    A set $S \subseteq [n] \setminus [k-1]$ is \emph{$k$-good} for $\F$ if, for every $F \in \F_{\ge k}$, we have $F \cap S \neq \emptyset$. We say $S$ is \emph{minimal} if none of its proper subsets are $k$-good.
\end{definition}

We emphasise that a $k$-good set $S$ does not need to be a member of the family $\F$. Furthermore, since $[n] \setminus [k-1]$ is $k$-good, we know that minimal $k$-good sets always exist. Since, in some sense, minimal $k$-good sets cover all the members $\mathcal{F}_{\ge k}$ efficiently, we might expect that they contain elements of large frequency, and we show that this is indeed the case.

\begin{proposition}\label{prop:medium}
Let $\mathcal{F}$ be a union-closed set family containing $m$ sets. If $k\geq 5$ and $2^k+2\leq m\leq 2^{3(k-1)}$, then $f_k(\mathcal{F})> \frac{1}{2^{k-1}+1}$.
\end{proposition}

\begin{proof}
    Let $\mathcal{F}$ be such a union-closed set family, and let $S \subseteq [n] \setminus [k-1]$ be minimally $k$-good for $\mathcal{F}$. By virtue of $k$-goodness, $S$ meets every set in $\mathcal{F}_{\ge k}$, and hence there are at least $|\mathcal{F}_{\ge k}|$ incidences between elements of $S$ and members of $\mathcal{F}_{\ge k}$. As there can be at most $2^{k-1}$ members of $\mathcal{F}$ whose support is contained in $[k-1]$, we have $|\mathcal{F}_{\ge k}| \ge | \mathcal{F} | - 2^{k-1} = m - 2^{k-1}$, and then averaging yields an element of $S$ contained in at least $\frac{m-2^{k-1}}{|S|}$ members of $\mathcal{F}$. In particular, since $S \subseteq [n] \setminus [k-1]$, and thus does not contain any of the $k-1$ most frequent elements, we have
    \begin{equation*}
    f_k(\mathcal{F}) \ge \frac{m-2^{k-1}}{m|S|}.
    \end{equation*}

    We thus need to bound the size of $S$. For this, we use the fact that $S$ is minimally $k$-good. Indeed, for every $y \in S$, we know that $S \setminus \{y\}$ is \emph{not} $k$-good for $\mathcal{F}$, and hence there must be some set $F_y \in \mathcal{F}_{\ge k}$ that is disjoint from $S \setminus \{y\}$. However, since $S$ \emph{is} $k$-good for $\mathcal{F}$, we know $F_y$ intersects $S$. Thus, we must have $F_y \cap S = \{y\}$.

    Given a subset $Y \subseteq S$, define $F_Y = \cup_{y \in Y} F_y$. Since $\mathcal{F}$ is union-closed, we must have $F_Y \in \mathcal{F}$. Moreover, since $F_Y \cap S = \cup_{y \in Y} (F_y \cap S) = Y$, these sets are all distinct. This shows $|\mathcal{F}| \ge 2^{|S|}$, or $|S| \le \log_2 m$. Plugging this into our lower bound for $f_k(\mathcal{F})$, we obtain
    \begin{equation*}
        f_k(\mathcal{F}) \ge \frac{m-2^{k-1}}{m \log_2 m}.
    \end{equation*}

    We can now deduce our result. Indeed, since $m > 2^k$, we have $m - 2^{k-1} > \frac12 m$, and thus $f_k(\mathcal{F}) > \frac{1}{2 \log_2 m} \ge \frac{1}{6(k-1)},$ where in the final inequality we use the upper bound $m \le 2^{3(k-1)}$. For all $k \ge 6$, we have $6(k-1) \le 2^{k-1} + 1$, and hence we have the desired bound.

    When $k = 5$, we have $\frac{1}{2 \log_2 m} \ge \frac{1}{17}$ for $m \le 2^{17/2}$, proving the conjecture in this case. On the other hand, if $m \ge 2^{17/2}$, then we actually have $m - 2^{k-1} > \frac{15}{16} m$, and thus our lower bound becomes $f_5(\mathcal{F}) \ge \frac{15}{16 \log_2 m} \ge \frac{5}{64} > \frac{1}{17}$, which again gives the required result.
\end{proof}
    
\section{Concluding remarks} \label{sec:conclusion}

In this paper, we studied a conjecture of Nagel regarding the frequency of the $k$th-most popular element in union closed set families. We used entropic methods to handle large families, while resorting to combinatorial arguments for smaller families. In this final section, we combine these results to establish our main result, and then address some outstanding open questions.

\subsection{Piecing it all together}

We proved three results --- \cref{prop:small,prop:medium,prop:large} --- that established the conjecture for union-closed set families $\F$ of various sizes $m = |\F|$. \cref{tab:range} summarises the different ranges in which these results apply.

\begin{table}[htbp!]
    \centering
    \begin{tabular}{c|c|c|c|c}
        & $k = 2$ & $k = 3$ & $k = 4$ & $k \geq 5$ \\ \hline
        Conjectured bound & $f_k(\mathcal{F})\geq\frac13$ & $f_k(\mathcal{F})\geq\frac15$ & $f_k(\mathcal{F})\geq\frac19$ & $f_k(\mathcal{F})\geq\frac1{2^{k-1}+1}$\\ \hline
        \cref{prop:small}& $m\leq 5$ & $m\leq 9$ & $m\leq 17$ & $ m \le 2^k + 1$ \\
        \cref{prop:medium}&  &  &  & $2^k + 2 \leq m \leq 2^{3(k-1)}$ \\
        \cref{prop:large}& $m\geq 4$ & $m\geq 6$ & $m\geq 14$ & $m \geq 2^{2.71(k-1)}$ \\
    \end{tabular}
    \caption{Effective range of three propositions, with $|\mathcal{F}|=m$}
    \label{tab:range}
\end{table}

We see that between them, the propositions settle all cases. Note that \cref{prop:large} and \cref{prop:medium} establish strict lower bounds for the conjectured bound on $f_k(\F)$, while \cref{prop:small} shows that the only case achieving equality is the near-$k$-cube. Thus, putting these results together yields a complete resolution of~\cref{conj:nagel} for $k \ge 2$.

\subsection{Frequencies in large families}

In fact, for large union-closed set families $\F$, we did not just show that $f_k(\F) > \frac{1}{2^{k-1} + 1}$, but proved quite a bit more. \cref{thm:largefamilies} indicates the following:

\begin{corollary}\label{cor:large}
    For any fixed $k$, $f_k(\mathcal{F})\geq \frac{3-\sqrt{5}}{2}-o(1)$ as $|\mathcal{F}|\to\infty$.
\end{corollary}

Note that the value $\frac{3 - \sqrt{5}}{2}$ is the constant obtained from the proof of Gilmer~\cite{gilmer2022constant}, with the calculations as optimised by Alweiss, Huang and Sellke~\cite{alweiss2022improved} and Pebody~\cite{pebody2022extension}. That is, there does not appear to be any discernible difference in the behaviour of the most frequent element and the $k$th-most frequent element. This leads one to wonder whether a direct analogue of Frankl's Conjecture (\cref{conj:UCSC}) might hold in large families for the $k$th-most frequent element; namely, that it should lie in half the sets. Our attempts to construct families whose $k$th frequency is significantly smaller have not borne fruit, and so we pose the following conjecture.

\begin{conjecture}\label{conj:Fkhalf}
    For any $k\in \N$, $f_k(\mathcal{F})=\frac12-o(1)$ when $|\mathcal{F}|\to \infty$.
\end{conjecture}

We note that, in contrast to~\cref{conj:UCSC}, this bound can only hold asymptotically. Indeed, the following example shows that $f_k(\F)$ cannot approach $\frac12$ \emph{too} quickly when $k \ge 2$.

\begin{example}\label{example:notconstant}

Let $\{S_i : 1 \le i \le k-1\}$ be a collection of pairwise-disjoint ground sets, each of size $n+1$. For each $i$, let $s_i^* \in S_i$ be a distinguished element.

Define $\F_i = \emptyset \cup \{ F \subseteq S_i : s_i^* \in F \}$, and take the set family $\F$ to be the \emph{direct sum} of these families, given by
\[ \F = \uplus \F_i = \{ \cup_{i = 1}^{k-1} F_i : F_i \in \F_i \}. \]
It is straightforward to verify that $\F$ is a union-closed set family of $m = (2^n+1)^{k-1}$ sets, in which the $k$th-most frequent element appears in just $2^{n-1} (2^n + 1)^{k-2}$ sets. Hence, we have
\[ f_k(\F) = \frac{2^{n-1}(2^n + 1)^{k-2}}{(2^n + 1)^{k-1}} = \frac12 - \frac{1}{2 m^{1/(k-1)}} < \frac12. \]
\end{example}

As a first step towards~\cref{conj:Fkhalf}, one could see if the improved bounds of Sawin~\cite{sawin2023improved}, Cambie~\cite{cambie2022better}, and Liu~\cite{liu2023improving} for the Union-Closed Sets Conjecture can also be employed in this setting to show a lower bound of $f_k(\F) \ge \frac{3 - \sqrt{5}}{2} + \delta$, for some $\delta > 0$, provided $\F$ is sufficiently large.

\section*{Acknowledgements}
    We are grateful to Kuo-Han Ku for some initial discussions, and to Hung-Hsun Hans Yu and Ting-Wei Chao for their comments on an early manuscript.

\bibliographystyle{amsplain}

\bibliography{reference}

\end{document}